\documentclass[11pt,oneside,reqno]{amsart}

\usepackage{amsmath,amsthm} 
\usepackage{enumerate}
\numberwithin{equation}{section}
\theoremstyle{definition}
\newtheorem{Definition}{Definition}[section]
\newtheorem{Example}[Definition]{Example}
\newtheorem{Remark}[Definition]{Remark}

\theoremstyle{plain}
\newtheorem{Theorem}[Definition]{Theorem}
\newtheorem{MainTheorem}{Theorem}

\newtheorem{Proposition}[Definition]{Proposition}

\newtheorem{Lemma}[Definition]{Lemma}

\usepackage{amssymb}        
\usepackage{mathrsfs}       
\usepackage{eucal}          
\usepackage{stmaryrd}       

\usepackage{geometry}

\usepackage[colorlinks,allcolors=blue]{hyperref}

\usepackage[usenames,dvipsnames]{xcolor}
\usepackage{tikz}
\usetikzlibrary{arrows,matrix,patterns}
\usepackage{tikz-cd}
\usepackage{subfig}         

\newcommand{\al}{\alpha}
\newcommand{\be}{\beta}

\newcommand{\Ga}{\Gamma}

\newcommand{\ep}{\varepsilon}

\newcommand{\ka}{\kappa}
\newcommand{\la}{\lambda}

\newcommand{\si}{\sigma}

\newcommand{\om}{\omega}

\newcommand{\N}{\mathbb{N}}
\newcommand{\Z}{\mathbb{Z}}

\newcommand{\R}{\mathbb{R}}
\newcommand{\C}{\mathbb{C}}

\newcommand{\T}{\mathbb{T}}

\newcommand{\Be}{\boldsymbol{e}}
\newcommand{\Bi}{\boldsymbol{i}}

\newcommand{\Fsl}{\mathfrak{sl}}

\newcommand{\CA}{\mathcal{A}}

\newcommand{\CI}{\mathcal{I}}

\newcommand{\SL}{\mathscr{L}}

\DeclareMathOperator{\End}{End}

\DeclareMathOperator{\Id}{Id}

\DeclareMathOperator{\ord}{ord}

\DeclareMathOperator{\sgn}{sgn}

\DeclareMathOperator{\Supp}{Supp}

\newcommand{\un}{\underline}

\renewcommand{\tilde}{\widetilde}

\title[Unitarizability of weight modules]{Unitarizability of weight modules over noncommutative Kleinian fiber products}
\author{Jonas T. Hartwig}
\date{}

\address{Department of Mathematics, Iowa State University, Ames, IA-50011, USA}
\email{jth@iastate.edu}
\urladdr{jthartwig.net}

\begin{document}

\begin{abstract}
Noncommutative Kleinian fiber products $\CA(\SL)$ are $\ast$-algebras attached to $(m,n)$-periodic higher spin six-vertex configurations $\SL$. Examples include quotients of the enveloping algebra of the affine Lie algebra $A_1^{(1)}$ and of a finite W-algebra of $\Fsl_4$. We determine which simple integral weight $\CA(\SL)$-modules carry an invariant inner product and which of them are unitarizable, and we compute the signature of the form. All three answers are read off from the vertex configuration: the signature of a module is a count of faces in a two-coloring of an eight-vertex configuration canonically attached to $\SL$. The main tool is a one-parameter family $\Delta_\xi$ of representations of $\CA(\SL)$ by difference operators on sections of a line bundle over the face lattice. These representations carry an explicit indefinite invariant inner product, and together they contain every simple integral weight module. Along the way we show that the Casimir element of $\CA(\SL)$ is unitary.
\end{abstract}

\maketitle

\section{Introduction}

\subsection{Invariant inner products}
By an inner product $\langle\cdot,\cdot\rangle$ on a complex vector space $V$ we shall mean in this paper a (not necessarily positive definite) non-degenerate symmetric sesquilinear form. An inner product on a module $V$ over a $\ast$-algebra $\CA$ is called invariant if $\langle av,w\rangle = \langle v,a^\ast w\rangle$ for all $a\in \CA$ and $v,w\in V$.

Modules which can be equipped with an invariant inner product are called \emph{pseudo-unitarizable}, and \emph{unitarizable} if the form can be chosen positive definite. For a pseudo-unitarizable module the next question is the \emph{signature} of the form, that is, the dimensions of maximal positive and negative definite subspaces. In many cases at most one invariant inner product exists up to a real multiple, for example for finite-dimensional indecomposable modules \cite{MazTur2001}, so the signature is an invariant of the module.

These questions matter on both sides of mathematical physics. Only unitarizable representations act on a Hilbert space, so only they can serve as state spaces of a quantum system. Representations with an indefinite invariant form, that is, representations on Krein spaces \cite{Bog1974}, arise for instance in covariant quantization of gauge fields, where the negative-norm states have to be identified and removed. In representation theory, unitarizability is a strong constraint. Every finite-dimensional representation of a compact group is unitarizable (see e.g.\ \cite[Prop.~4.6]{Kna1996}). The discrete series of unitary highest weight modules over the Virasoro algebra is a celebrated example \cite{KacRai1987}. Pseudo-unitarizable simple weight modules with finite-dimensional weight spaces over a semisimple Lie algebra were classified with respect to the Chevalley involution in \cite{MazTur2001b}.

\subsection{Noncommutative Kleinian fiber products}
In this paper we answer these questions completely for a family of $\ast$-algebras $\CA(\SL)$ called \emph{noncommutative Kleinian fiber products} \cite{HarRos2016,Har2016}. The input is an $(m,n)$-periodic higher spin six-vertex configuration $\SL$, that is, a configuration of lattice paths on a cylinder with twisted boundary conditions, as in integrable lattice models. The algebras $\CA(\SL)$ are noncommutative deformations of the algebra of functions on a fiber product of two type $A$ Kleinian singularities \cite{Har2016}. Examples include central extensions of the noncommutative Kleinian singularities of Hodges \cite{Hod1993}, and quotients of the enveloping algebra of the affine Lie algebra $A_1^{(1)}$ and of the finite W-algebra $\mathcal{W}(\Fsl_4,\Fsl_2\oplus\Fsl_2)$ \cite{Har2016}. They are rank two twisted generalized Weyl algebras \cite{MazTur1999}; $\ast$-representations of such algebras were studied in \cite{MazTur2002}.

In \cite{Har2016} the simple integral weight $\CA(\SL)$-modules were classified: they are the modules $M(D,\xi)$, where $D$ is a connected component of the cylinder minus the edges of $\SL$ and $\xi\in\C$, with $\xi=0$ if and only if $D$ is contractible (Theorem \ref{thm:weight}). It is therefore natural to ask which of them are pseudo-unitarizable or unitarizable, and what their signatures are. In rank one, that is, for generalized Weyl algebras such as the noncommutative type $A$ Kleinian singularities and $U_q(\Fsl_2)$ at roots of unity, the answer is given by sign conditions on the defining polynomial along the support of the module \cite{Har2011}. In rank two there are many paths between two weights in the support, and a priori the sign conditions along different paths could contradict each other. We show that they never do, and that the resulting signs are again encoded by a vertex configuration.

\subsection{Main results}
Fix an $(m,n)$-periodic higher spin six-vertex configuration $\SL$ and equip $\CA(\SL)$ with the involution $H^\ast=H$, $(X_i^\pm)^\ast=X_i^\mp$.

\begin{MainTheorem}[Realization] \label{thm:mainA}
For every $\xi\in\C^\times$ there is a representation $\Delta_\xi$ of $\CA(\SL)$ on the space $\Ga_0(L_\xi)$ of finitely supported sections of a complex line bundle $L_\xi$ over the face lattice $F$, in which $H$ acts by multiplication and the generators $X_i^\pm$ act by difference operators (Theorem \ref{thm:Deltaxi-rep}). Moreover:
\begin{enumerate}[{\rm (i)}]
\item $\Delta_\xi$ is completely reducible, and it is the direct sum of all $M(D,0)$ with $D$ contractible and all $M(D,\xi)$ with $D$ incontractible (Proposition \ref{prp:semi}). In particular every simple integral weight $\CA(\SL)$-module is a submodule of some $\Delta_\xi$.
\item If $|\xi|=1$, then $\Delta_\xi$ is pseudo-unitarizable, with invariant inner product
\[\langle f,g\rangle = \sum_{\la\in F} \tilde{f}(\la)\overline{\tilde{g}(\la)}\, \mathsf{w}(\la),\]
where $\mathsf{w}:F\to\{1,-1\}$ is an explicit sign function determined by $\SL$ (Theorem \ref{thm:Delta-pseudo}).
\end{enumerate}
\end{MainTheorem}

\begin{MainTheorem}[Pseudo-unitarizability] \label{thm:mainB}
The Casimir element $C$ of $\CA(\SL)$ from \cite{Har2016}, an element of a localization of $\CA(\SL)$, is unitary: $C^\ast C=1=CC^\ast$ (Proposition \ref{prp:C-is-unitary}). Consequently, $M(D,0)$ is pseudo-unitarizable for every contractible $D$, and for incontractible $D$ the module $M(D,\xi)$ is pseudo-unitarizable if and only if $|\xi|=1$ (Theorem \ref{thm:M-pseudo}). In both cases the invariant inner product is unique up to a nonzero real multiple.
\end{MainTheorem}
\newpage
\begin{MainTheorem}[Signature] \label{thm:mainC}
Let $M=M(D,\xi)$ be pseudo-unitarizable.
\begin{enumerate}[{\rm (i)}]
\item The signs of the polynomials $P_i^\SL$ on the internal edges of $D$ define an eight-vertex configuration $\SL^D$ in $D$ (Lemma \ref{lem:eight-vertex}).
\item The connected components of $D\setminus\overline{\SL^D}$ admit a two-coloring in which components sharing an edge of $\SL^D$ have different colors, and this coloring is unique up to interchanging the colors. If $c_\pm$ is the number of faces of $D$ lying in components of each color, then the signature of $M$ is $\{c_+,c_-\}$. In particular the signature does not depend on $\xi$ (Theorem \ref{thm:signature}).
\item $M$ is unitarizable if and only if $\SL^D$ is empty. Equivalently, for each internal edge $e$ of $D$, an even number of edges of $\SL$ parallel to $e$, counted with multiplicity, lie above the line through $e$ of slope $n/m$ (Proposition \ref{prp:unitarizable}).
\end{enumerate}
\end{MainTheorem}

\subsection{Discussion}
Three features of these results seem worth pointing out.

\emph{Positivity from combinatorics.} Unitarizability is a positivity property, but here it is decided by finite combinatorial data. The six-vertex configuration $\SL$ that defines the algebra produces, inside each region $D$, an eight-vertex configuration $\SL^D$, and the signature is a face count in its two-coloring. This continues the picture from \cite{Har2016}, where the center and the simple weight modules of $\CA(\SL)$ were read off from $\SL$.

\emph{A Gelfand model for all simple modules.} Theorem \ref{thm:mainA} realizes every simple integral weight module, finite- or infinite-dimensional, inside a single explicit family of difference-operator representations. The key input is Lemma \ref{lem:square-roots}: square roots of the polynomials $P_i^\SL$ can be chosen so that they still satisfy the Mazorchuk--Turowska equation. This consistency is not automatic. It is what makes $\Delta_\xi$ well defined, and it also produces the sign function $\mathsf{w}$ and the fundamental symmetry of the resulting Krein space (Remark \ref{rem:J}).

\emph{The role of the involution.} For the configuration in Example \ref{ex:affine}, related to evaluation modules of the affine Lie algebra $A_1^{(1)}$, the finite-dimensional modules $M(D,\xi)$ with $|\xi|=1$ are far from unitarizable for $\ast$: their signature is as balanced as possible. For the Chevalley-type involution $\dagger$ the same modules are unitarizable, which recovers the classical unitarizability of the corresponding loop modules. Theorem \ref{thm:mainC} detects this difference directly from the configuration.

\subsection{Outlook}
For unitarizable modules one can complete $\Ga_0(L_\xi)$ to a Hilbert space and study the resulting, possibly unbounded, $\ast$-representations in the spirit of \cite{MazTur2002}. It would also be interesting to know whether the square-root construction of Lemma \ref{lem:square-roots} extends to twisted generalized Weyl algebras of higher rank. Lastly, the problem of unitarizability for weight modules over rank two twisted generalized Weyl algebras with support in a finite (torus) orbit would also be interesting.

\subsection{Summary of paper}
Section \ref{sec:pre} reviews noncommutative Kleinian fiber products. Section \ref{sec:Delta} constructs the square roots and the representations $\Delta_\xi$, and Section \ref{sec:pseudo} proves that $\Delta_\xi$ is pseudo-unitarizable when $|\xi|=1$. Section \ref{sec:simple-weight} proves that the Casimir element is unitary, determines which simple integral weight modules are pseudo-unitarizable, and decomposes $\Delta_\xi$ into simple modules. Section \ref{sec:signature} computes the signatures and proves the unitarizability criterion. Section \ref{sec:examples} contains examples.

\section{Preliminaries}
\label{sec:pre}

\subsection{Noncommutative Kleinian fiber products}

Let $(m,n)$ be a pair of relatively prime non-negative integers, and $(\al_1,\al_2)=(\al,\be)\in\R^2\setminus\{(0,0)\}$ with $m\al+n\be=0$. Put
\begin{alignat}{2}
F &= \Z\al+\Z\be  &\qquad\qquad V &= F+(\al+\be)/2\\
E_i &= F+\al_i/2  &\qquad\qquad E &= E_1\cup E_2
\end{alignat}

\begin{Definition}
An \emph{$(m,n)$-periodic higher spin vertex configuration} $\SL=(\SL_1,\SL_2)$ is a pair of functions $\SL_i:E_i\to \N=\{0,1,2,\ldots\}$ with $|\SL_i^{-1}([1,\infty))|<\infty$ satisfying the \emph{current conservation rule}:
\begin{equation}\label{eq:CMTE}
\SL_1(v+\be/2)+\SL_2(v+\al/2)=\SL_1(v-\be/2)+\SL_2(v-\al/2)\qquad \text{for all $v\in V$}.
\end{equation}
\end{Definition}

Let $\tilde{\CA}=\tilde{\CA}(\SL)$ be the associative algebra generated by $\{H,X_1^+,X_1^-,X_2^+,X_2^-\}$
subject to defining relations
\begin{equation}\label{eq:rels}
[H,X_i^\pm]=\pm \al_i X_i^\pm\qquad  X_i^\pm X_i^\mp = P_i^\SL(H\mp\al_i/2)\qquad [X_1^\pm,X_2^\mp]=0
\end{equation}
where $[a,b]=ab-ba$ and 
\begin{equation}\label{eq:P-def}
P_i^\SL(u) = \prod_{e\in E_i} (u-e)^{\SL_i(e)}\qquad \text{for $i=1,2$}.
\end{equation}
Let $\CA=\CA(\SL)=\tilde{\CA}/\CI$ where 
\begin{equation}\label{eq:CI-def}
\CI=\{a\in \tilde{\CA}\mid \text{$p(H)a=0$ for some nonzero polynomial $p$}\}.
\end{equation}

\begin{Definition}
$\CA$ is the \emph{noncommutative Kleinian fiber product associated to $\SL$}.
\end{Definition}

Note that $(p_1,p_2)=(P_1^\mathscr{L},P_2^\mathscr{L})$ is a solution to the \emph{Mazorchuk-Turowska Equation (MTE)}
\begin{equation}\label{eq:MTE}
p_1(u+\al_2/2)p_2(u+\al_1/2)=p_1(u-\al_2/2)p_2(u-\al_1/2)
\end{equation}
which is necessary and sufficient for $\CA(\SL)$ to be nontrivial \cite[Prop.~1.11]{Har2016}.
Conversely, up to affine transformations, any solution $(p_1,p_2)$ to \eqref{eq:MTE} is a product of such lattice solutions $(P_1^\SL,P_2^\SL)$ \cite{HarRos2016,Har2016}.

\section{Realization by difference operators on line bundles}
\label{sec:Delta}

In this section we construct a natural family of representations $\Delta_\xi$ of $\CA(\mathscr{L})$ by difference operators acting on global sections of a complex line bundle $L_\xi$ over the one-dimensional lattice $F$. Later we show that every irreducible integral weight representation of $\CA(\mathscr{L})$ is a subrepresentation of $\Delta_\xi$ for appropriate $\xi$. Thus this provides a concrete realization of all simple integral weight modules.

The key result in the construction of $\Delta_\xi$, established in Section \ref{sec:sqrt}, is the existence of square roots (in fact, logarithms) of solutions to the MTE. The subtlety lies in proving that there exists a consistent choice of square roots $P_i^\SL(e)^{1/2}$ in such a way that the pair of functions still solve the MTE. This choice can be expressed combinatorially directly in terms of the configuration $\SL$ (Remark \ref{rem:l}) and gives rise to both the fundamental symmetry $J$ of the inner product space (Remark \ref{rem:J}) and formulas for the signature of the inner product given Section \ref{sec:signature}.

\subsection{\texorpdfstring{$\sqrt{\text{MTE}}$}{Square root of the MTE}} \label{sec:sqrt}
The following shows that solutions to the Mazorchuk-Turowska equation have square roots that are also solutions.

\begin{Lemma} \label{lem:square-roots}
There exists a pair of functions $(q_1,q_2)$, $q_i:E_i\to\C$ such that
\begin{enumerate}[{\rm (i)}]
\item $[q_i(e)]^2=P_i^\SL(e)$ for all $e\in E_i$ and $i\in\{1,2\}$,
\item $(q_1,q_2)$ is a solution to the MTE \eqref{eq:MTE}.
\end{enumerate}
\end{Lemma}
\begin{proof}
Put $p_i(u)=P_i^\SL(u)$ for $i=1,2$.
Consider the following pair of functions $(l_1,l_2)$:
\begin{equation}\label{eq:li-def}
l_i(u)=\sum_{e\in E_i,\, e>u} \mathscr{L}_i(e)\qquad \text{for $u\in E_i$ and $i=1,2$.}
\end{equation}
Here $>$ is the usual total order on $\R$.
We claim that these satisfy the following two properties:
\begin{equation}\label{eq:PvsL}
\frac{p_i(u)}{|p_i(u)|} = (-1)^{l_i(u)}
\end{equation}
and
\begin{equation}\label{eq:li-CMTE}
l_1(v+\al_2/2)+l_2(v+\al_1/2)=l_1(v-\al_2/2)+l_2(v-\al_1/2)\qquad \text{for all $v\in V$.}
\end{equation}
To check \eqref{eq:PvsL}, use the definition of $p_i(u)$ and that 
\begin{equation}
\frac{(u-e)^{\mathscr{L}_i(e)}}{|u-e|^{\mathscr{L}_i(e)}}=
\begin{cases}(-1)^{\mathscr{L}_i(e)}& e>u\\
1& e<u
\end{cases}
\end{equation}
To prove \eqref{eq:li-CMTE}, substituting \eqref{eq:li-def} into \eqref{eq:li-CMTE} and cancelling terms we obtain the following, where we assumed WLOG that $\al_1\in\Z_{<0}$ and $\al_2\in\Z_{>0}$:
\begin{equation}\label{eq:li-CMTE-2}
\sum_{\substack{e\in E_1\\ v-\al_2/2<e\le v+\al_2/2}}  \mathscr{L}_1(e) =
\sum_{\substack{e\in E_1\\ v+\al_1/2<e\le v-\al_1/2}}  \mathscr{L}_2(e)\qquad \text{for all $v\in V$.}
\end{equation}
Since this equation is additive in $\mathscr{L}$, we may without loss of generality assume that $\mathscr{L}$ consists of a single generalized Dyck path of period $(m,n)$. 
To this end, it is easy to verify \eqref{eq:li-CMTE-2} for the maximum area path consisting of $n$ steps of $\al_2$ followed by $m$ steps of $\al_1$. An induction argument shows that if $(l_1,l_2)$ solves \eqref{eq:li-CMTE} for a certain generalized Dyck path $\mathscr{L}$, then it also holds for the path in which a $21$ step has been replaced with a $12$ step. This proves the claim.
Now define
\begin{equation}\label{eq:q_i}
q_i(e)=\exp(2\pi\Bi \tfrac{l_i(e)}{4}) |p_i(e)|^{1/2}\qquad \text{for all $e\in E_i$ and $i=1,2$.}
\end{equation}
where $\Bi^2=-1$. Then $(q_1,q_2)$ satisfies the required properties.
\end{proof}

\begin{Remark}
Actually the proof shows one can take $N$:th roots of solutions too:
 \[p_i^{1/N}(e)=\exp(2\pi\Bi \tfrac{l_i(e)}{2N}) |p_i(e)|^{1/N}.\]
\end{Remark}

\begin{Remark}\label{rem:l}
The combinatorial interpretation of the functions $l_i:E_i\to\N$ is as follows. For each vertical edge $e\in E_1$, $l_1(e)$ counts the number (with multiplicity) of vertical edges in $\SL$ lying above the straight line through $e$ of slope $n/m$. Similarly for $l_2(e)$, $e\in E_2$, counting horizontal edges. (The ``location'' of an edge is by convention its midpoint.) See Figure \ref{fig:52-e}.
\end{Remark}

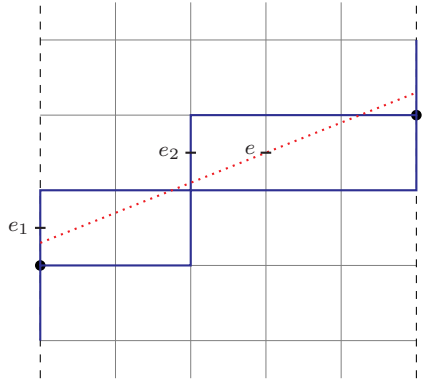
\begin{figure}
\centering
\begin{tikzpicture}
\foreach \y in {-1,0,...,3} {
\draw[help lines] (0 cm,\y cm) -- (5 cm,\y cm); }
\foreach \x in {1,...,4} {
\draw[help lines] (\x cm,-1.5 cm) -- (\x cm,3.5 cm); }
\draw[dashed] (0,-1.5 cm) -- (0,3.5cm);
\draw[dashed] (5,-1.5 cm) -- (5,3.5cm);
\fill (0,0) circle (2pt);
\fill (5,2) circle (2pt);
\draw[thick,Blue] (0,-1) -- (0,0) -- (2,0) -- (2,1) -- (5,1) -- (5,2);
\draw[thick,Blue] (0,0) -- (0,1) -- (2,1) -- (2,2) -- (5,2) -- (5,3);
\draw[thick,Red,dotted] (0, .3) -- (5,2.3); 
\draw[thick,Black] (3,1.55) node[font=\scriptsize, left] {$e$};
\draw[thick,Black] (3cm-2pt,1.5) -- (3cm+2pt,1.5);
\draw[thick,Black] (-2pt,.5) -- (0,.5) node[font=\scriptsize,left] {$e_1$} -- (2pt,.5);
\draw[thick,Black] (2cm-2pt,1.5) -- (2,1.5) node[font=\scriptsize,left] {$e_2$} -- (2cm+2pt,1.5);
\end{tikzpicture}
\caption{A fundamental domain for a $(5,2)$-periodic vertex configuration $\SL$ (solid blue). Here the vertical edge $e$  has $l_1(e)=2$ because there are two vertical edges, $e_1$ and $e_2$, appearing in $\SL$ above the line (dotted red) through $e$ of slope $5/2$. }
\label{fig:52-e}
\end{figure}

\subsection{Construction of the representation \texorpdfstring{$\Delta_\xi$}{xi}}
On the discrete space $F$ we define a complex line bundle $L_\xi$ as follows.
Fix $\xi\in\C^\times$. Let the abelian group $\Z$ act on $\Z^2\times \C$ by 
\begin{equation}
1.(a,b,z)=(a+m,b+n,\xi z)
\end{equation}
and define $L_\xi = (\Z^2\times\C)/\Z$ with bundle map $L_\xi\to F$ induced by $(x,y,z)\mapsto x\al+y\be$.
Let $\Ga(L_\xi)$ be the corresponding vector space of global sections.
Since $(x,y)\mapsto x\al+y\be$ induces a bijection $\Z^2/\langle(m,n)\rangle\simeq F$, we make the identification
\begin{equation}
\Ga(L_\xi)=\{f:\Z^2\to\C\mid \forall (x,y)\in\Z^2:\, f(x-m,y-n)=\xi\cdot f(x,y) \}.
\end{equation}
It is easy to see that $\Ga(L_\xi)$ consists of all functions of the form
\begin{equation}
f(x,y)=\tilde{f}(x\al+y\be)\exp\left(\frac{-xm-yn}{m^2+n^2}\log\xi\right)
\end{equation}
where $\tilde{f}:\Z\to\C$ is any function and $\log\xi\in\C$ is any choice of logarithm. Consider the subspace
\begin{equation}
\Ga_0(L_\xi)=\big\{f\in\Ga(L_\xi)\mid \text{$\tilde{f}$ has compact ($=$finite) support}\big\}.
\end{equation}
One checks that a $\C$-basis for $\Ga_0(L_\xi)$ is given by $\{f_\la\mid \la\in F\}$ where
\begin{equation}\label{eq:fk}
f_\la(x,y)=\delta_{x\al+y\be,\la}\exp\left(\frac{-xm-yn}{m^2+n^2}\log\xi\right)\qquad\text{for $\la\in F$}.
\end{equation}

The following theorem shows that the algebra $\CA(\mathscr{L})$ acts naturally on $\Ga_0(L_\xi)$ by difference-multiplication operators.

\begin{Theorem} \label{thm:Deltaxi-rep}
For each $\xi\in\C^\times$, there exists a representation
\begin{equation}
\Delta_\xi: \CA(\mathscr{L}) \to \End_\C\big(\Ga_0(L_\xi)\big)
\end{equation}
uniquely determined by 
\begin{subequations}\label{eq:Delta-def}
\begin{align}
\big(\Delta_\xi(X_1^\pm) f\big)(x,y) &= q_1(x \al+y\be \mp \al/2) \cdot f(x\mp 1,y), \\
\big(\Delta_\xi(X_2^\pm) f\big)(x,y) &= q_2(x \al+y\be \mp \be/2) \cdot f(x,y\mp 1), \\ 
\big(\Delta_\xi(H)f\big)(x,y) &= (x\al + y\be)\cdot f(x,y),
\end{align}
\end{subequations}
where the $q_i$ where defined in \eqref{eq:q_i}.
\end{Theorem}

\begin{proof}
For brevity, put $\tilde{X}_i^\pm =\Delta_\xi(X_i^\pm)$, $\tilde{H}=\Delta_\xi(H)$ and $p_i=P_i^\mathscr{L}$.
We have
\begin{align*}
\big(\tilde{X}_1^\pm \tilde{X}_1^\mp f\big)(x,y)
 &= q_1(x\al+y\be\mp\al/2)\cdot (\tilde{X}_1^\mp f)(x\mp 1,y)  \\
 &= q_1(x\al+y\be\mp\al/2)
    q_1((x\mp 1)\al+y\be\pm\al/2) \cdot f(x,y)  \\
 &= p_1(x\al+y\be\mp\al/2) \cdot f(x,y)  \\
 &= \big(p_1(\tilde{H}\mp\al/2)f\big)(x,y).
\end{align*}
This shows that $\tilde{X}_1^\pm \tilde{X}_1^\mp = p_1(\tilde{H}\mp\al/2)$.
Similarly one checks that $\tilde{X}_2^\pm \tilde{X}_2^\mp = p_2(\tilde{H}\mp\be/2)$.

Next, we verify that $[\tilde{H},\tilde{X}_i^\pm] = \pm \al_i \tilde{X}_i^\pm$ where $(\al_1,\al_2)=(\al,\be)$ for brevity. We have
\begin{align*}
\big([\tilde{H},\tilde{X}_1^\pm]f\big)(x,y)
 &= (x\al+y\be)\cdot (\tilde{X}_1^\pm f)(x,y) - q_1(x\al+y\be\mp \al/2) (\tilde{H}f)(x\mp 1,y) \\ 
 &= (x\al+y\be)q_1(x\al+y\be\mp \al/2) f(x\mp 1,y) \\
 & \quad - q_1(x\al+y\be\mp \al/2)  ((x\mp 1)\al+y\be) f(x\mp 1,y) \\
 &= \pm \al  q_1(x\al+y\be\mp \al/2) f(x\mp 1,y) \\ 
 &= (\pm \al \tilde{X}_1^\pm f)(x,y)
\end{align*}
and similarly for $\tilde{X}_2^\pm$.

The most crucial calculation is to verify that $[\tilde{X}_1^\pm,\tilde{X}_2^\mp]=0$ which is where we need that $(q_1,q_2)$ satisfy the MTE \eqref{eq:MTE}.
\begin{align*}
\big([\tilde{X}_1^+,\tilde{X}_2^-]f\big)(x,y) &=
(\tilde{X}_1^+ \tilde{X}_2^- f)(x,y) - (\tilde{X}_2^- \tilde{X}_1^+ f)(x,y) \\ 
&=q_1(x\al+y\be -\al/2)(\tilde{X}_2^- f)(x- 1,y) - q_2(x\al+y\be+\be/2)(\tilde{X}_1^+f)(x,y+1)\\
&=q_1(x\al+y\be-\al/2)q_2(x\al+y\be-\al+\be/2)f(x- 1, y+ 1) \\
&\quad- q_1(x\al+y\be - \al/2 + \be) q_2(x\al+y\be+\be/2) f(x- 1, y+ 1)\\
&=\big(q_1(v-\be/2)q_2(v-\al/2)-q_1(v+\be/2)q_2(v+\al/2)\big)f(x+ 1,y- 1)\\
&=0
\end{align*}
where we put $v=x\al+y\be - \al/2 +\be/2$. By $1\leftrightarrow 2$ the other case also holds.
This shows that \eqref{eq:Delta-def} defines a homomorphism $\Delta_\xi:\tilde{\CA}(\SL)\to\End_\C\big(\Ga_0(L_\xi)\big)$.

It remains to show that the torsion ideal $\CI$ in \eqref{eq:CI-def} is in the kernel of $\Delta_\xi$. For any $d=(d_1,d_2)\in\Z^2$ and any $a\in\tilde{\CA}(\SL)_d$, define the function $h_a:F\to\C$ by
\begin{equation} \label{eq:ha-def}
\big(\Delta_\xi(a)f\big)(x,y) = h_a(x\al+y\be)f(x-d_1,y-d_2)
\end{equation}
for all $f\in \Ga_0(L_\xi)$ and $(x,y)\in\Z^2$.
Before proceeding we pause to observe some properties of the functions $h_a$.

\begin{Lemma}\label{lem:ha}
Let $a$ and $b$ be homogeneous elements of $\tilde{\CA}(\SL)$. Let $(d_1,d_2)=\deg a$ and $(d_1',d_2')=\deg b$. Then the following properties hold for all $\la\in F$:
\begin{gather}
\label{eq:ha-1} h_{ab}(\la) = h_a(\la) h_b(\la-d_1\al-d_2\be), \\ 
\label{eq:ha-2} h_{a^\ast}(\la)=h_a(\la+d_1\al+d_2\be), \\
\label{eq:ha-3} h_{aa^\ast}(\la) = \big(h_a(\la)\big)^2.
\end{gather}
\end{Lemma}

\begin{proof}
\eqref{eq:ha-1}: By definition of $h_a$, \eqref{eq:ha-def}, we have
\[\big(\Delta_\xi(ab)f\big)(x,y) = h_{ab}(x\al+y\be) f(x-d_1-d_1',y-d_2-d_2').\]
On the other hand,
\begin{align*}
\big(\Delta_\xi(ab)f\big)(x,y) 
&= \big(\Delta_\xi(a)\Delta_\xi(b)f\big)(x,y) \\
&= h_a(x\al+y\be) \big(\Delta_\xi(b)f\big)(x-d_1,y-d_2)\\
&= h_a(x\al+y\be) h_b(x\al+y\be-d_1\al-d_2\be) f(x-d_1-d_1',y-d_2-d_2')
\end{align*}
Now choose $f=f_\mu$ as defined in \eqref{eq:fk} where $\mu=\la+(d_1+d_1')\al+(d_2+d_2')\be$ and cancel the exponential factor to obtain \eqref{eq:ha-1}.

\eqref{eq:ha-2}: Since this identity is linear in $a$, it suffices to prove it when $a$ is a product of elements of the set $\{X_1^+,X_1^-,X_2^+,X_2^-\}$. For this we use induction on the number of factors. If $a=X_i^\pm$ then $d=\pm\Be_i$ and \eqref{eq:ha-2} follows directly from the definition of $\Delta_\xi$, \eqref{eq:Delta-def}. For the induction step we write $a= X_i^\pm b$ and use \eqref{eq:ha-1},
noting that $\deg(b)=d\mp\Be_i$,
\begin{align*}
h_{a^\ast}(\la)=h_{b^\ast X_i^\mp}(\la) 
&\overset{\eqref{eq:ha-1}}{=} h_{b^\ast}(\la) h_{X_i^\mp}(\la+d_1\al_1+d_2\al_2\mp \al_i) \\
&\overset{\text{ind.hyp.}}{=} h_b(\la+d_1\al_1+d_2\al_2\mp \al_i) h_{X_i^\pm}(\la+d_1\al_1+d_2\al_2) \\
&\overset{\eqref{eq:ha-1}}{=} h_a(\la+d_1\al_1+d_2\al_2).
\end{align*}
Here we used the alternative notation $(\al_1,\al_2)=(\al,\be)$, and we switched the order of the two factors before applying \eqref{eq:ha-1} in the last equality. This proves the induction step.

\eqref{eq:ha-3}: Follows immediately by combining \eqref{eq:ha-1} and \eqref{eq:ha-2}, with $\deg(a^\ast)=-(d_1,d_2)$.
\end{proof}

Resuming the proof of Theorem \ref{thm:Deltaxi-rep}, recall that we must prove that $\CI$ is contained in the kernel of $\Delta_\xi$. Since $\CI$ is a graded ideal with respect to the $\Z^2$-gradation on $\tilde{\CA}(\SL)$ given by $\deg X_i^\pm=\pm\Be_i$, $\deg H=0$, it suffices to prove that if $d=(d_1,d_2)\in\Z^2$ and $a\in\tilde{\CA}(\SL)_d$ belongs to $\CI$, then $\Delta_\xi(a)=0$.
Since $a\in\CI$ we have $aa^\ast=0$. Thus $\Delta_\xi(aa^\ast)=0$, hence $h_{aa^\ast}(\la)=0$ for all $\la\in F$. By Lemma \ref{lem:ha}, $h_{aa^\ast}(\la)=\big(h_a(\la)\big)^2$ for all $\la\in F$. Therefore $h_a(\la)=0$ for all $\la\in F$. Thus $\Delta_\xi(a)=0$ as required.

This completes the proof of the existence of the homomorphism $\Delta_\xi$. The uniqueness follows from the fact that $\CA(\SL)$ is generated by the elements $X_i^\pm$ and $H$.
\end{proof}

\section{Pseudo-unitarizability of \texorpdfstring{$\Delta_\xi$}{Delta xi}} \label{sec:pseudo}

In Section \ref{sec:pseudo-general} we review the basic definitions needed for the following section, where we prove that $\Delta_\xi$ is pseudo-unitarizable when $|\xi|=1$.

\subsection{Pseudo-unitarizable modules over \texorpdfstring{$\ast$}{*}-algebras} \label{sec:pseudo-general}

In this subsection let $\CA$ denote a \emph{$\ast$-algebra}, by which we mean an associative unital algebra over $\C$ equipped with a conjugate-linear map $\CA\to \CA, a\mapsto a^\ast$ satisfying
\begin{equation}
(ab)^\ast = b^\ast a^\ast\qquad (a^\ast)^\ast = a\qquad \text{for all $a,b\in\CA$.}
\end{equation}

\begin{Definition} \label{def:inner-product}
Let $M$ be a module over $\CA$.
By an \emph{inner product} on $M$,
\[\langle \cdot,\cdot \rangle:M\times M\to \C\]
we mean a non-degenerate symmetric sesquilinear form:
\begin{enumerate}[{\rm (i)}]
\item $\langle \la u+\mu v, w\rangle=\la \langle u,w\rangle + \mu \langle v,w\rangle$ for all $u,v,w\in M$ and $\la,\mu\in \C$,
\item $\langle v,w\rangle = \overline{\langle w,v\rangle}$ for all $v,w\in M$, where the bar denotes complex conjugation,
\item if $w\in M$ and $\langle v,w\rangle=0$ for all $v\in M$,  then $w=0$.
\end{enumerate}
An inner product $\langle\cdot,\cdot\rangle$ on $M$ is called ($\ast$-)\emph{invariant} if
\begin{enumerate}[{\rm (i)}]
\setcounter{enumi}{3}
\item $\langle av,w\rangle = \langle v,a^\ast w\rangle$ for all $a\in\CA$, $v,w\in M$
\end{enumerate}
and \emph{positive definite} if
\begin{enumerate}[{\rm (i)}]
\setcounter{enumi}{4}
\item $\langle v,v\rangle>0$ for all nonzero $v\in V$.
\end{enumerate}
\end{Definition}

\begin{Definition} \label{def:unitar}
Let $M$ be an $\CA$-module. Then $M$ is \emph{pseudo-unitarizable} if there exists an invariant inner product on $M$ and \emph{unitarizable} if there exists a positive definite invariant inner product on $M$.
\end{Definition}

\begin{Definition} \label{def:dual}
The \emph{finitistic dual} of an $\CA$-module with a decomposition $M=\bigoplus_{\la\in\C} M_\la$, $\dim_\C M_\la<\infty$, is defined as
\[M^\#=\bigoplus_{\la\in\C} M^\#_\la,\qquad M^\#_\la=\{f:M_\la \to\C\mid \text{$f$ is conjugate-linear}\}\]
with $\CA$-action 
\[(af)(v)=f(a^\ast v),\quad\forall a\in \CA, f\in M^\#_\la, v\in M_\la, \la\in\C.\]
\end{Definition}

\begin{Theorem}\label{thm:pseudo-unitarizability}
\hspace{1em}
\begin{enumerate}[{\rm (a)}]
\item $M$ is pseudo-unitarizable if and only if $M^\#\simeq M$.
\item If $M$ is indecomposable there is at most one invariant inner product on $M$, up to equivalence.
\end{enumerate}
\end{Theorem}
\begin{proof} Follows from general results in \cite{MazTur2001}.
\end{proof}

\subsection{Pseudo-unitarizability of \texorpdfstring{$\Delta_\xi$}{Delta xi}}

The algebras $\CA(\SL)$ become $\ast$-algebras by defining
\begin{equation} \label{eq:AL-star}
H^\ast = H\qquad (X_i^\pm)^\ast = X_i^\mp\qquad \text{for $i=1,2$.}
\end{equation}
In this section we give an explicit invariant inner product on the representation space $\Ga_0(L_\xi)$.
We assume $|\xi|=1$ and write $\xi=e^{2\pi\Bi\ka}$. Every $f\in\Ga_0(L_\xi)$ has the form
\begin{equation}
f(x,y) = \tilde{f}(x\al+y\be)\exp\left(2\pi\Bi \frac{-xm-yn}{m^2+n^2}\ka\right)
\end{equation}
for some unique function $\tilde{f}:F\to\C$ of finite support.

\begin{Theorem} \label{thm:Delta-pseudo} $\Delta_\xi$ is pseudo-unitarizable. More precisely, there exists a weight function $\mathsf{w}:F\to\{+1,-1\}$ such that
\begin{equation} \label{eq:form-def}
\langle f,g\rangle = \sum_{\la\in F} \tilde{f}(\la)\overline{\tilde{g}(\la)} \mathsf{w}(\la)
\end{equation}
is a binary form on $\Ga_0(L_\xi)$ satisfying (i)--(iv) of Definition \ref{def:unitar}.
\end{Theorem}

\begin{proof}
The form is invariant iff $\langle \Delta_\xi(X_i^\pm)f,g\rangle=\langle f,\Delta_\xi(X_i^\mp)g\rangle$ because $\CA(\SL)$ is generated by $X_i^\pm$ and $H$, and that $\langle \Delta_\xi(H)f,g\rangle = \langle f,\Delta_\xi(H)g\rangle$ holds is immediate because $\Delta_\xi(H)$ is diagonal with real eigenvalues. We have
\begin{align*}
\langle \Delta_\xi(X_1^\pm)f,g\rangle &= \sum_{\la\in F} q_1(\la\mp \tfrac{\al}{2})\tilde{f}(\la\mp\al)\overline{\tilde{g}(\la)}\mathsf{w}(\la) \\
&=\sum_{\la\in F} q_1(\la\pm\tfrac{\al}{2})\tilde{f}(\la)\overline{\tilde{g}(\la\pm\al)} \mathsf{w}(\la\pm\al)\\
&=\sum_{\la\in F} \tilde{f}(\la)\overline{q_1(\la\pm\tfrac{\al}{2})\tilde{g}(\la\pm\al)}
\left(\frac{q_1(\la\pm\tfrac{\al}{2})}{|q_1(\la\pm\tfrac{\al}{2})|}\right)^2 \mathsf{w}(\la\pm \al)
\end{align*}
and similarly for $X_2^\pm$ and $\be$ which leads to the conditions
\begin{equation}\label{eq:suff-cond-w}
\left(\frac{q_i(\la\pm\tfrac{\al_i}{2})}{|q_i(\la\pm\tfrac{\al_i}{2})|}\right)^2 \mathsf{w}(\la\pm \al_i)=\mathsf{w}(\la)\qquad \text{for $i=1,2$.}
\end{equation}
If \eqref{eq:suff-cond-w} hold then the form $\langle\cdot,\cdot\rangle$ defined by \eqref{eq:form-def} is invariant.
Substituting \eqref{eq:q_i} into \eqref{eq:suff-cond-w} we obtain
\begin{equation} \label{eq:w-diff-eq}
\exp\left(2\pi\Bi l_i(\la\pm\tfrac{\al_i}{2})/2\right) \mathsf{w}(\la\pm\al_i)=\mathsf{w}(\la),\qquad i=1,2
\end{equation}
Again it suffices to take a phase function
\begin{equation}\label{eq:w-def}
\mathsf{w}(\la)=\exp\left(2\pi\Bi \om(\la)\right)
\end{equation}
The system of difference equations for $\om(\la)$ can then be written
\begin{equation}
\om(\la\pm\al_i) \equiv_\Z \om(\la) + \frac{1}{2} l_i(\la\pm\tfrac{\al_i}{2}),\qquad i=1,2
\end{equation}
where $a\equiv_\Z b$ iff $a-b\in\Z$.
For this system to have solutions the $l_i$ must satisfy  consistency equations which can
be written
\begin{equation} \label{eq:l-cong}
\frac{1}{2} l_1(\la-\tfrac{\al_2}{2})+\frac{1}{2}l_2(\la-\tfrac{\al_1}{2}) \equiv_\Z
\frac{1}{2} l_1(\la+\tfrac{\al_2}{2})+\frac{1}{2}l_2(\la+\tfrac{\al_1}{2})
\end{equation}
which actually holds as an equality due to the current conservation \eqref{eq:li-CMTE}.
This proves that the system of difference equations is consistent and with boundary condition $\omega(0)=0$ we obtain the unique solution
\begin{equation} \label{eq:omega-solution}
\om\left(\pm(\al_{i_1}+\al_{i_2}+\cdots+\al_{i_k})\right)=
\frac{1}{2}\sum_{r=1}^k l_{i_r}\left(\pm(\al_{i_1}+\al_{i_2}+\cdots+\al_{i_{r-1}}+\tfrac{\al_{i_r}}{2})\right)
\end{equation}
for any binary sequence $\un{i}=i_1i_2\ldots i_k\in\mathsf{Seq}_2$. Since $\Z\al_1+\Z\al_2=F$ and $m\al_1+n\al_2=0$ there are non-negative integers $a,b$ such that $F=\langle a\al_1+b\al_2\rangle$ as abelian groups, proving that the elements $\pm(\al_{i_1}+\cdots+\al_{i_k})$ run through all of $F$. Due to \eqref{eq:l-cong} the value of $\om$ modulo $\Z$ is independent of $\un{i}$. This gives a unique solution $\mathsf{w}(\la)$ to \eqref{eq:w-diff-eq} having $\mathsf{w}(0)=1$.

That this form is symmetric $\langle f,g\rangle=\overline{\langle g,f\rangle}$ follows from the fact that $\mathsf{w}(\la)$ is real-valued. Actually $\mathsf{w}(\la)\in\{1,-1\}$ for all $\la\in F$ because $l_i(\la)$ are integer valued hence $\omega(\la)\in\frac{1}{2}\Z$.

Finally $\langle\cdot,\cdot\rangle$ is non-degenerate: If $\langle f,g\rangle=0$ for all $g$, we can pick $g=f_\la$, see \eqref{eq:fk}. Then $\tilde{f_\la}(\mu)=\delta_{\la,\mu}$ where $\delta$ is Kronecker's delta, hence
\[\langle f,f_\la\rangle = \pm \tilde{f}(\la),\qquad \forall \la\in F\]
hence $\tilde{f}(\la)=0$ for all $\la\in F$ which implies that $f$ is identically zero.
\end{proof}

\begin{Remark}
This means we have produced an explicit isomorphism
$\Ga_0(L_\xi)\simeq \Ga_0(L_\xi)^\#$, namely
$f\mapsto \langle f,\cdot\rangle$.
\end{Remark}

\begin{Remark}
When $|\xi|=1$, this gives the following independent proof that $\Delta_\xi(\CI)=0$.
Since $\CI$ is a graded ideal with respect to the $\Z^2$-grading on $\widetilde{\CA}(\SL)$ it suffices to prove that $\CI_d\in\ker\Delta_\xi$ for each $d\in\Z^2$. Let $a\in\CI_d$. Then $a^\ast a=0$ by \cite[Thm.~3.11(ii)$\Rightarrow$(i)]{Har2016}. Thus we have for any $\la\in F$,
\[\langle \Delta_\xi(a)f_\la,\Delta_\xi(a)f_\la\rangle =
\langle \Delta_\xi(a^\ast a)f_\la,f_\la\rangle = 0.
\]
Since the form is non-degenerate, and all weight spaces are one-dimensional and pairwise orthogonal, there are no nonzero  isotropic weight vectors. This implies that $\Delta_\xi(a)=0$.
\end{Remark}

\begin{Remark}\label{rem:J}
In the language of \cite{Bog1974}, the \emph{fundamental symmetry} operator for the (indefinite) inner product space $\Ga_0(L_\xi)$,
\[J:\Ga_0(L_\xi)\to \Ga_0(L_\xi)\]
is given by
\[J f_\la = \mathsf{w}(\la) f_\la \qquad\text{for all $\la\in F$.}\]
and the $J$-eigenspace decomposition of $\Ga_0(L_\xi)$
\[\Ga_0(L_\xi)=\Ga_0(L_\xi)^+\oplus \Ga_0(L_\xi)^-\]
is the \emph{fundamental decomposition}. On the $+1$ (respectively $-1$) eigenspace the form $\langle\cdot,\cdot\rangle$ is positive (respectively negative) definite.
\end{Remark}

\section{Relation to simple integral weight modules} \label{sec:simple-weight}

In this section we prove that the representations $\Delta_\xi:\mathcal{A}(\mathscr{L})\to\End_\C\big(\Ga_0(L_\xi)\big)$ are completely reducible, and that every simple integral weight $\mathcal{A}(\mathscr{L})$-module is isomorphic to a submodule of $\Ga_0(L_\xi)$ for some $\xi$.

\subsection{Unitarity of the Casimir} \label{sec:casimir}

In \cite{Har2016} an $\CA(\SL)$-centralizing element of the localization $\CA(\SL)_{\mathrm{loc}}=\CA(\SL)=\otimes_{\C[H]}\C(H)$ was given.

\begin{Theorem}[{\cite[Prp.~6.3, Thm.~C]{Har2016}}] \label{thm:C}
Consider the element $C\in \CA(\mathscr{L})_{\mathrm{loc}}$ given by
\begin{equation}\label{eq:C}
C=X(\un{i})\prod_{\la\in F} (H-\la)^{-\ord(\un{i},\la)}
\end{equation}
where $\un{i}=i_1i_2\ldots i_{m+n}\in\mathsf{Seq}_2(m,n)$ is a sequence of $m$ $1$'s and $n$ $2$'s in any order, $X(\un{i})=X_{i_{m+n}}^+\cdots X_{i_2}^+X_{i_1}^+$, and $\ord(\un{i},\la)$ is the number (with multiplicity) of vertical (equivalently, horizontal) edges in $\SL$ intersected by the face lattice path $\la, \la+\al_{i_1}, \la+\al_{i_1}+\al_{i_2}, \ldots, \la+\al_{i_1}+\al_{i_2}+\cdots +\al_{i_{m+n}}=\la$.
\begin{enumerate}[{\rm (a)}]
\item $C$ is independent of the choice of $\un{i}\in\mathsf{Seq}_2(m,n)$.
\item $C$ belongs to the center of $\CA(\mathscr{L})_\mathrm{loc}$. In particular, $[C,\CA(\mathscr{L})]=0$.
\item $C\in\CA(\SL)$ iff $\SL$ is a five-vertex configuration (every vertex has at most two incident edges with nonzero multiplicity) in which case $Z(\CA(\SL))=\C[C,C^{-1}]$. Otherwise $Z\big(\CA(\SL)\big)=\C$.
\end{enumerate}
\end{Theorem}

\begin{Definition}
We call the element $C$ the \emph{Casimir element} for $\CA(\SL)$.
\end{Definition}

In \cite{Har2016} it was observed that $C^\ast C$ is a constant ($C^\ast C$ has $\Z^2$-degree $(0,0)$ and thus is a polynomial in $H$, but it also commutes with $X_i^\pm$ and is therefore a constant, a priori some rational number). In this subsection we prove that in fact this number equals $1$, i.e. $C$ is unitary. This further shows that $C$ is some canonical object.

First we prove two lemmas.

\begin{Lemma}\label{lem:L-ord}
For any $\la\in F$ and $\un{i}\in\mathsf{Seq}_2(m,n)$ we have (putting $\ell=\ell(\un{i})=m+n$),
\begin{equation}
\sum_{j=1}^\ell \SL_{i_j}(\la+\al_{i_1}+\al_{i_2}+\cdots+\al_{i_{j-1}}+\tfrac{\al_{i_j}}{2})=2\ord(\un{i},\la)
\end{equation}
\end{Lemma}

\begin{proof}
The left hand side counts the total number (with multiplicity) of  edges in $\SL$ that the face path $(\un{i},\la)$ crosses. Since $(\un{i},\la)$ is a loop, the number of horizontal edges in $\SL$ that $(\un{i},\la)$ crosses is the same as the number of vertical edges it crosses, and this number is exactly equal to the order of the path $(\un{i},\la)$, see \cite[Lem.~5.11]{Har2016}.
\end{proof}

Put
\begin{equation}\label{eq:quni}
q_{\un{i}}(H) = q_{i_1}(H+\tfrac{\al_{i_1}}{2})q_{i_2}(H+\al_{i_1}+\tfrac{\al_{i_2}}{2})\cdots q_{i_{\ell}}(H+\al_{i_1}+\al_{i_2}+\cdots+\al_{i_{\ell-1}}+\tfrac{\al_{i_\ell}}{2})
\end{equation}
and similarly for $p_{\un{i}}(H)$.
Remarkably, even though each $q_j(H)$ defined in \eqref{eq:q_i} is locally a branch of the square root of the polynomial $P_j^{\SL}(H)$, the following lemma shows that $q_{\un{i}}(H)$ is a polynomial, provided the sequence $\un{i}$ consists of $m$ 1's and $n$ 2's in any order.

\begin{Lemma} \label{lem:q-ord}
For any $\un{i}\in\mathsf{Seq}_2(m,n)$ we have, as functions on $F$,
\begin{equation}\label{eq:q-ord}
q_{\un{i}}(H) = \prod_{\la\in F} (H-\la)^{\ord(\un{i},\la)}.
\end{equation}
\end{Lemma}

\begin{proof}
Put $p_i(H)=P_i^\SL(H)$ and $\ell=\ell(\un{i})=m+n$. Then
\begin{equation}\label{eq:qlem-pf0}
\big(q_{\un{i}}(H)\big)^2 =
\prod_{j=1}^\ell p_{i_j}(H+\al_{i_1}+\cdots+\al_{i_{j-1}}+\tfrac{\al_{i_j}}{2}).
\end{equation}
Now note that
\begin{equation}\label{eq:qlem-pf1}
 p_k(H+\tfrac{\al_k}{2}) = \prod_{e\in E_k} (H+\tfrac{\al_k}{2}-e)^{\SL_k(e)}
\end{equation}
since $E_k=F+\al_k/2$ we make the substitition $\la=e-\al_k/2$ to rewrite \eqref{eq:qlem-pf1} as
\begin{equation}\label{eq:qlem-pf2}
p_k(H+\tfrac{\al_k}{2}) = \prod_{\la\in F} (H-\la)^{\SL_k(\la+\tfrac{\al_2}{2})}
\end{equation}
Applying \eqref{eq:qlem-pf2} to each factor in \eqref{eq:qlem-pf0} we get
\begin{align*}
p_{\un{i}}(H)&=\prod_{j=1}^\ell \prod_{\la\in F} 
(H+\al_{i_1}+\cdots+\al_{i_{j-1}}-\la)^{\SL_{i_j}(\la+\tfrac{\al_{i_j}}{2})}\\
&=\prod_{\la\in F} (H-\la)^{\sum_{j=1}^\ell \SL_{i_j}(\la+\al_{i_1}+\cdots+\al_{i_{j-1}}+\tfrac{\al_{i_j}}{2})}\\
&=\prod_{\la\in F} (H-\la)^{2\ord(\un{i},\la)}
\end{align*}
where we used Lemma \ref{lem:L-ord} in the last step. 

It remains to show both sides have the same sign. When evaluating at $H=\mu\in F$, both sides of \eqref{eq:q-ord} have the same zero set, namely the set of all $\mu\in F$ such that $\ord(\un{i},\mu)>0$. So it suffices to show that both sides have the same sign when they are nonzero. In fact we will show that both sides are nonnegative at all $\mu\in F$.
In the left hand side we have by \eqref{eq:quni} and \eqref{eq:q_i},
\begin{align*}
q_{\un{i}}(H) &=\prod_{j=1}^\ell q_{i_j}(H+\al_{i_1}+\cdots+\al_{i_{j-1}}+\tfrac{\al_{i_j}}{2}) \\
&=\exp\left(2\pi\Bi\sum_{j=1}^\ell l_{i_j}(H+\al_{i_1}+\cdots+\al_{i_{j-1}}+\tfrac{\al_{i_j}}{2})/4\right) \\
&\quad\cdot \prod_{j=1}^\ell |p_{i_j}(H+\al_{i_1}+\cdots+\al_{i_{j-1}}+\tfrac{\al_{i_j}}{2})|^{1/2}
\end{align*}
Setting $H=\mu$ in the exponential expression we get after dividing by $2\pi\Bi/2$
\[\frac{1}{2}\sum_{j=1}^\ell l_{i_j}\left(\mu+\al_{i_1}+\cdots+\al_{i_{j-1}}+\tfrac{\al_{i_j}}{2}\right).\]
If $\mu=\al_{k_1}+\cdots+\al_{k_r}$ for some $\un{k}=k_1k_2\cdots k_r\in\mathsf{Seq}_2$, then using \eqref{eq:omega-solution} this equals
\[\omega(\mu+\al_{i_1}+\cdots+\al_{i_\ell})-\omega(\mu)=0\]
since $\sum_{j=1}^\ell \al_{i_j}=m\al+n\be=0$.
A similar argument can be made in the case of $\mu=-(\al_{k_1}+\cdots+\al_{k_r})$. This proves that
\begin{equation}
q_{\un{i}}(\mu)\ge 0\qquad \text{for all $\mu\in F$.}
\end{equation}
Next we prove that the same is true for the product in the right hand side of \eqref{eq:q-ord}. Assuming $\mu$ is not a zero we have
\begin{equation}\label{eq:sgn-prod}
\sgn \left( \prod_{\la\in F} (\mu-\la)^{\ord(\un{i},\la)}\right)=
\sgn\left(
\prod_{\la\in F, \la>\mu} (\mu-\la)^{\ord(\un{i},\la)}\right)
=(-1)^{\sum_{\la\in F, \la>\mu} \ord(\un{i},\la)}.
\end{equation}
Using Lemma \ref{lem:L-ord} we get
\[
 \sum_{\substack{\la\in F \\ \la>\mu}} \ord(\un{i},\mu) 
 = \sum_{\substack{\la\in F \\ \la>\mu}} \frac{1}{2}\sum_{j=1}^\ell \SL_{i_j}\left(\la+\al_{i_1}+\cdots+\al_{i_{j-1}}+\tfrac{\al_{i_j}}{2}\right).
 \]
Interchanging the order of summation and making the change of variables $e=\la+\al_{i_1}+\cdots+\al_{i_{j-1}}+\tfrac{\al_{i_j}}{2}\in E_{i_j}$ we obtain
\[\frac{1}{2}\sum_{j=1}^\ell \sum_{\substack{e\in E_{i_j}\\ e>\mu+\al_{i_1}+\cdots+\al_{i_{j-1}}+\tfrac{\al_{i_j}}{2}}} \SL_{i_j}(e).\]
Now use the definition \eqref{eq:li-def} of $l_j$ to get
\[\frac{1}{2}\sum_{j=1}^\ell l_{i_j}\left(\mu+\al_{i_1}+\cdots+\al_{i_{j-1}}+\tfrac{\al_{i_j}}{2}\right)\]
which as shown above equals zero. This proves that for any $\un{i}\in\mathsf{Seq}_2(m,n)$,
\begin{equation}
\sum_{\substack{\la\in F\\ \la>\mu}} \ord(\un{i},\la) = 0 \qquad\text{for all $\mu\in F$ such that $\ord(\un{i},\mu)=0$}
\end{equation}
and hence by \eqref{eq:sgn-prod},
\begin{equation}
\prod_{\la\in F}(\mu-\la)^{\ord(\un{i},\la)}\ge 0\qquad\text{for all $\mu\in F$.}
\end{equation}
This finishes the proof of the identity \eqref{eq:q-ord}.
\end{proof}

We now prove that $C$ given in \eqref{eq:C} is unitary.

\begin{Proposition} \label{prp:C-is-unitary}
The Casimir element for $\CA(\SL)$ is unitary with respect to $\ast$. That is:
\begin{equation}\label{eq:C-is-unitary}
C^\ast \cdot C = 1 = C\cdot C^\ast
\end{equation}
\end{Proposition}

\begin{proof} Since $H^\ast=H$ and $\la\in F=\Z$ we have
\[C^\ast C = \prod_{\la\in F}(H-\la)^{-\ord(\un{i},\la)} X(\un{i})^\ast X(\un{i}) \prod_{\la\in F} (H-\la)^{-\ord(\un{i},\la)}
=\prod_{\la\in F} (H-\la)^{-2\ord(\un{i},\la)} X(\un{i})^\ast X(\un{i})\]
By a straightforward calculation (see e.g. proof of \cite[Lem.~5.5]{Har2016})
\[ X(\un{i})^\ast X(\un{i}) = p_{\un{i}}(H) \]
where $p_{\un{i}}(H)$ is as in \eqref{eq:quni} with $p_i=P_i^{\SL}$.
By Lemma \ref{lem:q-ord},
\[p_{\un{i}}(H)=(q_{\un{i}}(H))^2 =  \prod_{\la\in F}(H-\la)^{2\ord(\un{i},\la)}.\]
This finishes the proof.
\end{proof}

\subsection{Pseudo-unitarizablity of simple integral weight modules} \label{sec:pseudo-simple}

We recall the classification of simple integral weight $\CA(\SL)$-modules from \cite{Har2016}.

Consider the space $\T_{m,n}=\R^2/\langle (m,n)\rangle$ equipped with the quotient topology. This space is homeomorphic to a doubly infinite cylinder. Let $\overline{\SL}\subseteq \T_{m,n}$ be the configuration $\SL$ regarded as a union of closed line segments. Let $\overline{\mathsf{F}}\subseteq \T_{m,n}$ be the image of $\Z^2$ under the canonical projection $\R^2\to\T_{m,n}$.

\begin{Theorem}[{\cite[Thm.~B]{Har2016}}] \label{thm:weight}
\hspace{1em}
\begin{enumerate}[{\rm (a)}]
\item There is a bijective correspondence between the set of  isoclasses of simple integral weight $\CA(\SL)$-modules, and the set of pairs $(D,\xi)$ where $D$ is a connected component of $\T_{m,n}\setminus\overline{\mathscr{L}}$ and $\xi\in\C$ with $\xi=0$ iff $D$ is contractible.
\item Let $M(D,\xi)$ be the module corresponding to $(D,\xi)$. Each nonzero weight space $M(D,\xi)_\la$ is one-dimensional and 
\[\Supp\big(M(D,\xi)\big)=\{x_1\al_1+x_2\al_2\mid (x_1,x_2)+\langle(m,n)\rangle\in \overline{\mathsf{F}}\cap D\}.\]
\item For any incontractible $D$ and $\xi\in\C^\times$ the action of the Casimir element $C$ for $\CA(\SL)$ from \eqref{eq:C} is well-defined on $M(D,\xi)$ and $C|_{M(D,\xi)}=\xi\Id_{M(D,\xi)}$.
\end{enumerate}
\end{Theorem}

The following lemma is immediate because the support of an integral weight module is contained in $F$ which is a subset of $\R$.

\begin{Lemma} \label{lem:support-of-finitistic-dual}
If $M$ is a simple integral weight $\CA(\SL)$-module then $\Supp(M)=\Supp(M^\#)$.
\end{Lemma}

Using the unitarity of the Casimir $C$ from Proposition \ref{prp:C-is-unitary}, we obtain the following description of the pseudo-unitarizable simple integral weight $\CA(\SL)$-modules.

\begin{Theorem} \label{thm:M-pseudo}
Let $D$ be a connected component of $\T_{m,n}\setminus\overline{\SL}$.
\begin{enumerate}[{\rm (i)}]
\item If $D$ is contractible, then $M(D,0)$ is pseudo-unitarizable.
\item If $D$ is incontractible then $M(D,\xi)$ is pseudo-unitarizable if and only if $|\xi|=1$.
\end{enumerate}
\end{Theorem}

\begin{proof}
By Theorem \ref{thm:pseudo-unitarizability}, a simple weight module $M$ is pseudo-unitarizable if and only if $M^\#\simeq  M$.

(i) Put $M=M(D,0)$. By  Theorem \ref{thm:weight}, since $\Supp(M^\#)=\Supp(M)=D$ which is contractible, it follows that $M^\#\simeq M$ hence $M$ is pseudo-unitarizable.

(ii) By Theorem \ref{thm:weight} and Lemma \ref{lem:support-of-finitistic-dual}, for any $\xi\in\C^\times$ there exists $\xi^\#\in\C^\times$ such that $M(D,\xi)^\#\simeq M(D,\xi^\#)$. Recall that $\xi$ has the interpretation as being the eigenvalue of $C$. By Proposition \ref{prp:C-is-unitary}, $C^\ast=C^{-1}$ and thus for any $f\in M^\#$ and $v\in M$,
\[ (C f)(v)=f(C^\ast v) = f(\xi^{-1}v)=(\bar\xi^{-1}f)(v).\]
This proves that $\xi^\#=\bar\xi^{-1}$. By the classification theorem again, $M(D,\xi)\simeq M(D,\bar\xi^{-1})$ if and only if $\xi=\bar\xi^{-1}$ or equivalently, $|\xi|=1$.
\end{proof}

\subsection{Decomposition of \texorpdfstring{$\Ga_0(L_\xi)$}{Gamma0(L xi)} into irreducibles} \label{sec:semi}

Put
\begin{equation}
M_0=\bigoplus_D M(D,0)\qquad M_\xi = \bigoplus_{D'} M(D',\xi)
\end{equation}
where $D$ (respectively $D'$) runs over the set of contractible (respectively incontractible) connected components of $\T_{m,n}\setminus\overline\SL$, and $\xi\in\C^\times$ is fixed.

\begin{Proposition} \label{prp:semi}
For any $\xi\in\C^\times$ there is an isomorphism of $\CA(\SL)$-modules 
\begin{equation}
\Ga_0(L_{\xi}) \simeq M_0\oplus M_\xi.
\end{equation}
\end{Proposition}

\begin{proof}
Each $H$-weight space of $\Ga_0(L_\xi)$ is one-dimensional, spanned by $f_k$, $k\in\Z$. 
For each connected component $D$ of $\T_{m,n}\setminus\overline{\SL}$, there is a submodule of $\Ga_0(L_\xi)$ whose support is exactly $D$.
By the characterizing properties of the simple integral weight modules from Theorem \ref{thm:weight}, it remains to prove that if $D$ is an incontractible component and $f_k$ is one of the basis vectors where $k\in F(D)$, 
then $\Delta_\xi(C) f_k = \xi f_k$. 

Let $\un{i}\in\mathsf{Seq}_2(m,n)$. We have 
\begin{multline*}
\big(\Delta_\xi(X(\un{i}))f\big)(x,y) = 
\big(\Delta_\xi (X_{i_\ell}\cdots X_{i_1})f\big)(x,y) \\
\shoveleft{
=q_{i_\ell}(x\al+y\be-\tfrac{\al_{i_\ell}}{2})q_{i_{\ell-1}}(x\al+y\be-\al_{i_\ell}-\tfrac{\al_{i_{\ell-1}}}{2})\cdots }\\ 
\shoveright{
\cdots 
q_{i_1}(x\al+y\be-\al_{i_\ell}-\al_{i_{\ell-1}}-\cdots-\al_{i_2}-\tfrac{\al_{i_1}}{2})\cdot
f((x,y)-\Be_{i_\ell}-\cdots-\Be_{i_1}) }\\
=\big(q_{\un{i}}(\tilde{H}) f\big)(x-m,y-n)
=\xi q_{\un{i}}(\tilde{H}) \cdot f(x,y)
\end{multline*}
We have shown that as operators on $\Ga_0(L_\xi)$ we have
\begin{equation} \label{eq:delta-on-Xi}
\Delta_\xi(X(\un{i})) = \xi q_{\un{i}}(\tilde{H})
\end{equation}
where $\tilde{H}=\Delta_\xi(H)$.
Consider the Casimir element $C\in\CA(\SL)_{\mathrm{loc}}$ given by \eqref{eq:C}. We have 
\begin{gather}
\Delta_\xi(C) = \Delta_\xi\left(X(\un{i})\prod_{\la\in F}(H-\la)^{-\ord(\un{i},\la)}\right) = 
\xi q_{\un{i}}(\tilde H) \prod_{\la\in F} (\tilde{H}-\la)^{-\ord(\un{i},\la)} = \xi
\end{gather}
where we used Lemma \ref{lem:q-ord} in the last step.
We abused notation by applying $\Delta_\xi$ to an element of the localization, but the resulting operator is well-defined on any $f_\la$ for $\la$ in an incontractible component.
This finishes the proof.
\end{proof}

\section{On the signature of the invariant inner product on \texorpdfstring{$M(D,\xi)$}{M(D,xi)} and internal eight-vertex configurations} \label{sec:signature}

\subsection{Unitarizable simple integral weight \texorpdfstring{$\CA(\SL)$}{A(L)}-modules}
By Theorem \ref{thm:weight} each simple integral weight $\CA(\SL)$-module $M$ is isomorphic to $M(D,0)$ with $D$ contractible or $M(D,\xi)$ with $D$ incontractible and $\xi\in\C^\times$. 
As we saw in Theorem \ref{thm:M-pseudo}, the former are always pseudo-unitarizable and the latter iff $|\xi|=1$.
In this case there is a unique up to nonzero real multiple admissible form $\langle \cdot,\cdot\rangle$ on $M$ by Theorem \ref{thm:pseudo-unitarizability}. Thus, identifying $M$ with a submodule of $\Ga_0(L_\xi)$ as in Proposition \ref{prp:semi}, we may define the \emph{signature of $M$} to be 
\begin{equation}
\si(M) = \{\dim M^+, \dim M^-\}
\end{equation}
where $M^\pm$ are the $\pm 1$ eigenspaces of the fundamental symmetry $J$ from Remark \ref{rem:J}. All this just amounts to the following formula
\begin{equation}\label{eq:s-def}
\sigma(M)=\{s_+,s_-\}\qquad 
s_\pm = \#\{\la\in F\mid \text{$M_\la\neq 0$ and $\pm\langle v,v\rangle\ge 0$ for all $v\in M_\la $}\}.
\end{equation}
for some choice of invariant inner product $\langle\cdot,\cdot\rangle$. Changing the form to $-\langle\cdot,\cdot\rangle$ does not change $\sigma(M)$ as a set. Thus $\sigma(M)$ depends only on $M$ and not on the choice of invariant inner product on $M$.
We say that $M$ is \emph{definite} if $0\in \sigma(M)$. Thus $M$ is definite if and only if $M$ is unitarizable.

\begin{Lemma} \label{lem:signature}
The sign of the quadratic form $v\mapsto \langle v,v\rangle$ is the same on two adjacent weight spaces of weights $\la$ and $\la+\al_i$ if and only if $P_i^{\SL}(e)>0$ where $e=\la+\al_i/2$ is the (midpoint of the) edge separating the weight spaces.
\end{Lemma}

\begin{figure}
\centering
\begin{tikzpicture}
\node[font=\scriptsize, below] at (0,0) {$\la$};
\draw (-2pt,-2pt) -- ( 2pt, 2pt);
\draw ( 2pt,-2pt) -- (-2pt, 2pt);
\node[font=\scriptsize, below] at (1,0) {$\la+\al_1$};
\draw (1cm-2pt,-2pt) -- (1cm+2pt, 2pt);
\draw (1cm+2pt,-2pt) -- (1cm-2pt, 2pt);
\draw[Red] (.5,.5) -- (.5,-.5);
\end{tikzpicture}
\caption{Adjacent weight spaces in the support of a simple integral weight module.}
\label{fig:Psign}
\end{figure}

\begin{proof}
Suppose that $v_\la\in V$ is a nonzero eigenvector of $H$ with eigenvalue $\la$. Then
\begin{equation}
\label{eq:P-norm-ratio}
 \langle X_i^+ v_\la,\, X_i^+ v_\la\rangle 
 = \langle X_i^- X_i^+ v_\la,\, v_\la\rangle 
 = \langle P_i^\SL\big(H+ \frac{\al_i}{2}\big)v_\la,\, v_\la\rangle 
 = P_i^\SL\big(\la+\frac{\al_i}{2}\big)\langle v_\la, v_\la\rangle.
\end{equation}
\end{proof}

Recall the cylinder $\T_{m,n}=\R^2/\langle(m,n)\rangle$. For any subset $D\subseteq \T_{m,n}$ we put
\begin{equation}
E_i(D) = \{x\al+y\be\in E_i \mid (x,y)\in (\Z^2+\tfrac{1}{2}\Be_i)\cap D\}.
\end{equation}
where $\Be_1=(1,0)$ and $\Be_2=(0,1)$. Thus $E_1(D)$ (respectively $E_2(D)$) is the set of vertical (respectively horizontal) edges that when drawn in a fundamental domain in $\R^2$ have their midpoint inside $D$. Similarly we put
\begin{gather*}
F(D)=\{x\al+y\be\in F\mid (x,y)\in \Z^2\cap D\}, \\
V(D)=\{x\al+y\be\in V\mid (x,y)\in \big(\Z^2+\tfrac{1}{2}(\Be_1+\Be_2)\big)\cap D\}.
\end{gather*}
We call elements of these sets \emph{internal} edges, faces and vertices in $D$.

The following gives a characterization of unitarizable simple integral weight $\CA(\SL)$-modules.
\begin{Proposition} \label{prp:unitarizable}
Let $M(D,\xi)$ be a pseudo-unitarizable simple integral weight $\CA(\SL)$-module. Then the following statements are equivalent.
\begin{enumerate}[{\rm (i)}]
\item $M(D,\xi)$ is unitarizable.
\item $\mathsf{w}|_{F(D)}$ is constant where $\mathsf{w}$ was defined in \eqref{eq:w-def},\eqref{eq:omega-solution}.
\item $P_i^{\SL}(e_i)>0$ for all $e_i\in E_i(D)$ and $i\in\{1,2\}$.
\item $l_i(e_i)\in 2\N$ for all $e_i\in E_i(D)$ and $i\in\{1,2\}$, where $l_i$ was defined in \eqref{eq:li-def}.
\item At each internal vertical (respectiely horizontal) edge $e$ in $D$, there are an even number, counted with multiplicity, of vertical (respectively horizontal) edges occurring in $\SL$ whose midpoints are above the line through the midpoint of $e$ of slope $n/m$.
\end{enumerate}
\end{Proposition}

\begin{proof}
(i)$\Leftrightarrow$(ii) was noted above. (ii)$\Leftrightarrow$(iii) follows from Lemma \ref{lem:signature}. (iii)$\Leftrightarrow$(iv) is immediate by \eqref{eq:PvsL}. Lastly (iv)$\Leftrightarrow$(v) follows from Remark \ref{rem:l}.
\end{proof}

\subsection{Internal eight-vertex configurations and the signature of \texorpdfstring{$M(D,\xi)$}{M(D,xi)}} \label{sec:eight}

We turn to the final problem of calculating the signature of $M(D,\xi)$ as defined in the previous subsection.
Define the \emph{internal vertex configuration in $D$} to be $\SL^D=(\SL^D_1,\SL^D_2)$ where $\SL^D_i: E_i(D)\to \{1,-1\}$ is given by
\begin{equation}
\SL_i^D(e) = \sgn P_i^\SL(e)\qquad\text{for $e\in E_i$ and $i=1,2$.}
\end{equation}
Note that $P_i^{\SL}(e)\neq 0$ at every $e\in E_i(D)$. We interpret the value $+1$ as the edge being absent in the configuration $\SL^D$, and $-1$ as an edge present of multiplicity one. In figures edges in $\SL^D$ will be drawn dashed in red. We make the following observation.

\begin{Lemma} \label{lem:eight-vertex}
Let $D$ be any connected component of $\T_{m,n}\setminus \overline{\mathscr{L}}$. Then $\SL^D$ is an eight-vertex configuration in $D$. That is, at each internal vertex $v\in V(D)$, there are exactly eight possible local configurations, see Figure \ref{fig:eight-vertex}. 
\end{Lemma}

\begin{proof}
At an internal vertex $v$, both sides of
\[
P_1^\SL(v+\al_2/2)
P_2^\SL(v+\al_1/2)=
P_1^\SL(v-\al_2/2)
P_2^\SL(v-\al_1/2)
\]
are nonzero. Taking signs on both sides the claim follows.
\end{proof}

\begin{figure}[h]
\centering 
\begin{tikzpicture}
\begin{scope}[color=Red,style=dashed,thick]
\draw ( 0,-4) -- ( 0,-2);
\draw (-1,-3) -- ( 1,-3);
\draw ( 2, 0) -- ( 4, 0);
\draw ( 3,-4) -- ( 3,-2);
\draw ( 6,-1) -- ( 6, 0) -- ( 7, 0);
\draw ( 5,-3) -- ( 6,-3) -- ( 6,-2);
\draw ( 9, 1) -- ( 9, 0) -- (10, 0);
\draw ( 8,-3) -- ( 9,-3) -- ( 9,-4);
\end{scope}
\fill ( 0, 0) circle (2pt);
\fill ( 0,-3) circle (2pt);
\fill ( 3 ,0) circle (2pt);
\fill ( 3,-3) circle (2pt);
\fill ( 6, 0) circle (2pt);
\fill ( 6,-3) circle (2pt);
\fill ( 9, 0) circle (2pt);
\fill ( 9,-3) circle (2pt);
\end{tikzpicture}
\caption{Local eight-vertex configurations.}
\label{fig:eight-vertex}
\end{figure}
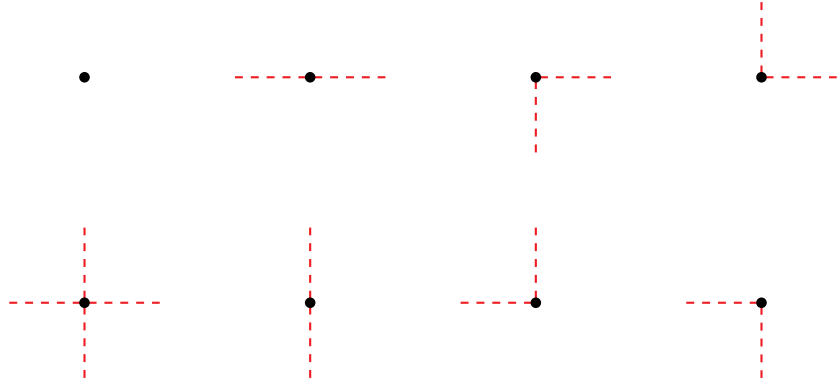

We now compute the signature of $M(D,\xi)$. Removing the union $\overline{\SL^D}$ of the line segments corresponding to edges of $\SL^D$ breaks $D$ into connected components $D^{(j)}$:
\begin{equation}\label{eq:subcomp}
D\setminus \overline{\SL^D}=\bigsqcup_j D^{(j)}.
\end{equation}
We call two components $D^{(j)}$ and $D^{(j')}$ \emph{adjacent} if some edge of $\SL^D$ has $D^{(j)}$ on one side and $D^{(j')}$ on the other.

\begin{Theorem} \label{thm:signature}
Let $D$ be a connected component of $\T_{m,n}\setminus\overline{\SL}$ and let $M=M(D,\xi)$ be pseudo-unitarizable, that is, $\xi=0$ if $D$ is contractible and $|\xi|=1$ otherwise.
\begin{enumerate}[{\rm (i)}]
\item An internal edge $e\in E_i(D)$ belongs to $\SL^D$ if and only if $l_i(e)$ is odd, that is, if and only if an odd number of edges of $\SL$ parallel to $e$, counted with multiplicity, have their midpoints above the line through the midpoint of $e$ of slope $n/m$.
\item No edge of $\SL^D$ has the same component $D^{(j)}$ on both sides, and the components $D^{(j)}$ admit a two-coloring in which adjacent components have different colors. This coloring is unique up to interchanging the two colors.
\item Let $c_+$ and $c_-$ be the numbers of faces $\la\in F(D)$ lying in components of each color. Then
\[\sigma\big(M(D,\xi)\big)=\{c_+,\,c_-\}.\]
In particular, $\sigma\big(M(D,\xi)\big)$ does not depend on $\xi$.
\end{enumerate}
\end{Theorem}

\begin{proof}
We first record an elementary fact about the square lattice on $\T_{m,n}$. Let $S$ be a set of edges and $\overline{S}$ the union of the corresponding closed line segments. Then two faces lie in the same connected component of $\T_{m,n}\setminus\overline{S}$ if and only if they can be joined by a sequence of faces in which any two consecutive faces are separated by an edge not in $S$. Indeed, $\T_{m,n}\setminus\overline{S}$ is the union of the open unit squares centered at the faces, the open edges not in $S$, and the vertices not in $\overline{S}$. An open edge not in $S$ joins the two squares it separates, and at a vertex not in $\overline{S}$ none of the four incident edges lies in $S$, so the four squares around it are already joined through these edges. In particular, every connected component of $\T_{m,n}\setminus\overline{S}$ contains a face.

Let $\langle\cdot,\cdot\rangle$ be an invariant inner product on $M$, which exists by Theorem \ref{thm:M-pseudo}. By Theorem \ref{thm:weight}(b), every nonzero weight space $M_\la$ is one-dimensional and the set of weights of $M$ is $F(D)$. Since $H^\ast=H$ and $F\subset\R$, we have $(\la-\mu)\langle v,w\rangle=\langle Hv,w\rangle-\langle v,Hw\rangle=0$ for $v\in M_\la$ and $w\in M_\mu$, so distinct weight spaces are orthogonal. As the form is non-degenerate and the weight spaces are one-dimensional, $\langle v,v\rangle\neq 0$ for every nonzero $v\in M_\la$. Hence
\[\ep(\la)=\sgn\langle v,v\rangle,\qquad v\in M_\la\setminus\{0\},\]
is a well-defined function $\ep:F(D)\to\{1,-1\}$, and by \eqref{eq:s-def},
\begin{equation}\label{eq:s-eps}
\sigma(M)=\{s_+,s_-\},\qquad s_\pm=\#\{\la\in F(D)\mid \ep(\la)=\pm1\}.
\end{equation}

Now let $e\in E_i(D)$ be an internal edge and put $\la=e-\al_i/2$. Since the midpoint of $e$ lies in $D$, $e$ is not an edge of $\SL$. Hence the faces $\la$ and $\la+\al_i$ separated by $e$ both lie in $F(D)$, and $\SL_i(e)=0$, so $P_i^\SL(e)\neq 0$ by \eqref{eq:P-def}. By \eqref{eq:P-norm-ratio}, for nonzero $v\in M_\la$ the vector $X_i^+v\in M_{\la+\al_i}$ satisfies $\langle X_i^+v,X_i^+v\rangle=P_i^\SL(e)\langle v,v\rangle\neq 0$. Therefore
\begin{equation}\label{eq:eps-flip}
\ep(\la+\al_i)=\sgn P_i^\SL(e)\cdot\ep(\la)=\SL_i^D(e)\,\ep(\la).
\end{equation}
In words: $\ep$ changes sign exactly across the edges of $\SL^D$.

(i) By \eqref{eq:PvsL}, $\SL_i^D(e)=\sgn P_i^\SL(e)=(-1)^{l_i(e)}$, so $e$ belongs to $\SL^D$ if and only if $l_i(e)$ is odd. The geometric description of $l_i(e)$ is Remark \ref{rem:l}.

(ii) Each component $D^{(j)}$ is a connected component of $\T_{m,n}\setminus(\overline{\SL}\cup\overline{\SL^D})$. By the fact recorded above, applied with $S$ the set of edges of $\SL$ and of $\SL^D$, any two faces in $D^{(j)}$ are joined by a sequence of faces in which consecutive faces are separated by internal edges of $D$ not in $\SL^D$. By \eqref{eq:eps-flip}, $\ep$ is constant along such a sequence, hence constant on $F(D^{(j)})$, which is nonempty. If $e$ is an edge of $\SL^D$, then by \eqref{eq:eps-flip} $\ep$ takes different values on the two faces separated by $e$, so these faces lie in different components. Hence no edge of $\SL^D$ has the same component on both sides, and coloring each $D^{(j)}$ by the value of $\ep$ on it gives a two-coloring in which adjacent components have different colors.
For uniqueness, apply the fact recorded above with $S$ the set of edges of $\SL$: any two faces of $D$ are joined by a sequence of faces in which consecutive faces are separated by internal edges of $D$. Consecutive faces lie in the same component or in adjacent components, so the graph whose vertices are the components $D^{(j)}$, with edges joining adjacent components, is connected. A two-coloring of a connected graph in which adjacent vertices have different colors is unique up to interchanging the colors.

(iii) For the coloring by $\ep$ from (ii) we have $c_\pm=s_\pm$, so $\sigma(M)=\{c_+,c_-\}$ by \eqref{eq:s-eps}. By the uniqueness in (ii), the set $\{c_+,c_-\}$ does not depend on the choice of coloring. Finally, $\SL^D$ is defined by the signs of the polynomials $P_i^\SL$, which do not involve $\xi$.
\end{proof}

\begin{Remark} \label{rem:local-global}
By Lemma \ref{lem:eight-vertex}, an even number of edges of $\SL^D$ meet at each internal vertex of $D$, which gives a two-coloring near each vertex. When $D$ is incontractible, this local condition does not by itself give a global two-coloring: a loop in $D$ around the cylinder could a priori cross an odd number of edges of $\SL^D$. Theorem \ref{thm:signature}(ii) rules this out. In the proof, this comes from the existence of the invariant inner product (Theorem \ref{thm:M-pseudo}).
\end{Remark}

In practice, Theorem \ref{thm:signature} gives the following procedure, which is used in Section \ref{sec:examples}. Color an internal edge $e$ of $D$ red if an odd number of edges of $\SL$ parallel to $e$ lie above the line through $e$ of slope $n/m$. Remove the red edges, two-color the resulting regions, and count the faces of each color.

\section{Examples} \label{sec:examples}

\begin{Example} \label{ex:affine}
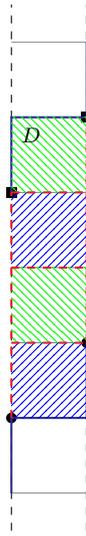
\begin{figure}[h]
\centering
\begin{tikzpicture}
\foreach \y in {0,1,...,6} {
\draw[help lines] (0 cm,\y cm) -- (1 cm,\y cm); }
\draw[dashed] (0,-.5 cm) -- (0,6.5cm);
\draw[dashed] (1,-.5 cm) -- (1,6.5cm);
\fill (0,1) circle (2pt);
\fill (1,2) circle (2pt);
\fill (0cm-2pt,4cm-2pt) rectangle (0cm+2pt,4cm+2pt);
\fill (1cm-2pt,5cm-2pt) rectangle (1cm+2pt,5cm+2pt);
\draw[thick,Blue] (0,0) -- (0,1) -- (1,1) -- (1,2);
\draw[thick,Blue] (0,4) -- (0,5) -- (1,5) -- (1,6);
\fill[pattern=north east lines, pattern color=Blue] (0,1) rectangle (1,2);
\fill[pattern=north east lines, pattern color=Blue] (0,3) rectangle (1,4);
\fill[pattern=north west lines, pattern color=Green] (0,2) rectangle (1,3);
\fill[pattern=north west lines, pattern color=Green] (0,4) rectangle (1,5);
\draw[thick,Red,dashed] (0,1) -- (0,4);
\draw[thick,Red,dashed] (1,2) -- (1,5);
\draw[thick,Red,dashed] (0,2) -- (1,2);
\draw[thick,Red,dashed] (0,3) -- (1,3);
\draw[thick,Red,dashed] (0,4) -- (1,4);
\node[font=\scriptsize, below right] at (0,5) {$D$}; 
\end{tikzpicture}
\caption{A $(1,1)$-periodic six-vertex configuration with two paths.}
\label{fig:11-2}
\end{figure}
Figure \ref{fig:11-2} shows a fundamental domain for $\T_{1,1}=\R^2/\langle(1,1)\rangle$ with the blue solid edges constituting the $d=4$ case of the $(1,1)$-periodic configuration $\SL$ with
\[P_1^{\SL}(u)=P_2^{\SL}(u)=\big(u-\frac{1}{2}\big)\big(u-\frac{1}{2}-d\big)\]
where $d$ is a positive integer. It was shown in \cite{Har2016} that the corresponding noncommutative Kleinian fiber product $\CA(\SL)$ is related to $d$-dimensional evaluation modules for the affine Lie algebra $A_1^{(1)}$ and to a finite W-algebra associated to $\Fsl_4$. $\T_{1,1}\setminus\overline{\SL}$ consists of three connected components, one of which is finite and denoted $D$. Since $D$ has the homotopy type of a circle, there is a one-parameter family of $d$-dimensional simple integral weight $\CA(\SL)$-modules $M(D,\xi)$. Coloring the internal edges as in Theorem \ref{thm:signature}(i), we obtain that all internal edges are red.
Thus if $|\xi|=1$ then $M(D,\xi)$ is pseudo-unitarizable with signature
\[
\si(M(D,\xi))=\begin{cases}\{d/2,d/2\}& \text{$d$ even}\\
\{(d-1)/2, (d+1)/2\}& \text{$d$ odd}
\end{cases}
\]
If we instead of $\ast$ consider the ``Chevalley'' involution $\dagger$ on $\CA(\SL)$ given by $(X_i^\pm)^\dagger = -X_i^\mp$, $H^\dagger=H$, then this is equivalent to changing signs of the polynomials $P_i^\SL(u)$, hence $\SL^D$ changes into $-\SL^D$, meaning now all internal edges in $D$ are transparent. This recovers the well-known unitarizability of these loop modules regarded as modules over the affine Lie algebra $A_1^{(1)}$.
\end{Example}

\begin{Example}
\begin{figure}[h]
\centering
\begin{tikzpicture}
\foreach \y in {-1,0,...,2} {
\draw[help lines] (0 cm,\y cm) -- (5 cm,\y cm); }
\foreach \x in {1,...,4} {
\draw[help lines] (\x cm,-1.5 cm) -- (\x cm,2.5 cm); }
\draw[dashed] (0,-1.5 cm) -- (0,2.5cm);
\draw[dashed] (5,-1.5 cm) -- (5,2.5cm);
\fill (0,0) circle (2pt);
\fill (5,2) circle (2pt);
\draw[thick,Blue] (0,-1) -- (0,-.5pt) -- (2,-.5pt) -- (2,1) -- (5,1) -- (5,2);
\draw[thick,Blue] (0,.5pt) -- (1,.5pt) -- (1,1) -- (2,1) -- (2,2) -- (5,2);
\fill[pattern=north east lines, pattern color=Blue] (3,1) rectangle (5,2);
\fill[pattern=north west lines, pattern color=Green] (2,1) rectangle (3,2);
\draw[thick,Red,dashed] (3,2) -- (3,1);
\node[font=\scriptsize, below right] at (1,1) {$D_1$}; 
\node[font=\scriptsize, below right] at (2,2) {$D_2$};
\end{tikzpicture}
\caption{A fundamental domain for a $(5,2)$-periodic configuration $\SL$ consisting of two vertex paths. $M(D_1,0)$ is unitarizable, while $M(D_2,0)$ has signature $\{1,2\}$.}
\label{fig:52-2}
\end{figure}
Consider the $(5,2)$-periodic higher spin vertex configuration $\SL$ in Figure \ref{fig:52-2} consisting of the two vertex lattice paths $1121112$ and $1212111$ ($1$ being a step right and $2$ being a step up) with the same starting point. Removing these line segments from the doubly infinte cylinder $\T_{5,2}=\R^2/\langle(5,2)\rangle$ leaves four connected components, two of which are finite, $D_1$ of area $1$ and $D_2$ of area $3$. Thus the noncommutative Kleinian fiber product $\CA(\SL)$ has exactly two finite-dimensional simple modules, the one-dimensional $M(D_1,0)$ and the three-dimensional $M(D_2,0)$. The former is unitarizable (since there are no internal edges to check) while the latter has one red internal vertical edge, meaning an edge where $P_1(e)<0$. By Theorem \ref{thm:signature} this implies the signature of $M(D_2,0)$ is $\{1,2\}$.
\end{Example}

\begin{Example}
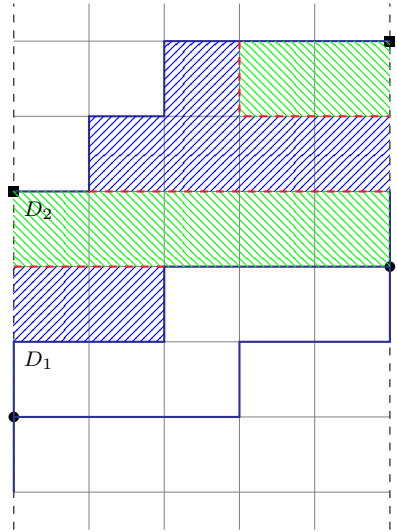
\begin{figure}[h]
\centering
\begin{tikzpicture}
\foreach \y in {-1,0,...,5} {
\draw[help lines] (0 cm,\y cm) -- (5 cm,\y cm); }
\foreach \x in {1,...,4} {
\draw[help lines] (\x cm,-1.5 cm) -- (\x cm,5.5 cm); }
\draw[dashed] (0,-1.5 cm) -- (0,5.5cm);
\draw[dashed] (5,-1.5 cm) -- (5,5.5cm);
\fill (0,0) circle (2pt);
\fill (5,2) circle (2pt);
\fill (0cm-2pt,3cm-2pt) rectangle (0cm+2pt,3cm+2pt);
\fill (5cm-2pt,5cm-2pt) rectangle (5cm+2pt,5cm+2pt);
\draw[thick,Blue] (0,-1) -- (0,0) -- (3,0) -- (3,1) -- (5,1) -- (5,2);
\draw[thick,Blue] (0,0) -- (0,1) -- (2,1) -- (2,2) -- (5,2) -- (5,3);
\draw[thick,Blue] (0,3) -- (1,3) -- (1,4) -- (2,4) -- (2,5) -- (5,5);
\fill[pattern=north east lines, pattern color=Blue] (0,1) rectangle (2,2);
\fill[pattern=north east lines, pattern color=Blue] (1,3) rectangle (5,4);
\fill[pattern=north east lines, pattern color=Blue] (2,4) rectangle (3,5);
\fill[pattern=north west lines, pattern color=Green] (0,2) rectangle (5,3);
\fill[pattern=north west lines, pattern color=Green] (3,4) rectangle (5,5);
\draw[thick,Red,dashed] (0,2) -- (2,2);
\draw[thick,Red,dashed] (1,3) -- (5,3);
\draw[thick,Red,dashed] (3,5) -- (3,4) -- (5,4);
\node[font=\scriptsize, below right] at (0,1) {$D_1$}; 
\node[font=\scriptsize, below right] at (0,3) {$D_2$}; 
\end{tikzpicture}
\caption{A $(5,2)$-periodic example with $\SL$ consisting of three vertex paths. $M(D_1,0)$ is definite, however $M(D_2,\xi)$ has signature $\{7,7\}$ for each $\xi\in\C^\times$, $|\xi|=1$.}
\label{fig:52-3}
\end{figure}
In Figure \ref{fig:52-3}, $\CA(\SL)$ has one $6$-dimensional simple module $M(D_1,0)$ and a one-parameter family of $14$-dimensional simple modules $M(D_2,\xi)$, $\xi\in\C^\times$. Using Theorem \ref{thm:signature} one checks that the module $M(D_1,0)$ is unitarizable. For $|\xi|=1$ the module $M(D_2,\xi)$ is pseudo-unitarizable of signature $\{7,7\}$.
\end{Example}

\begin{Example}
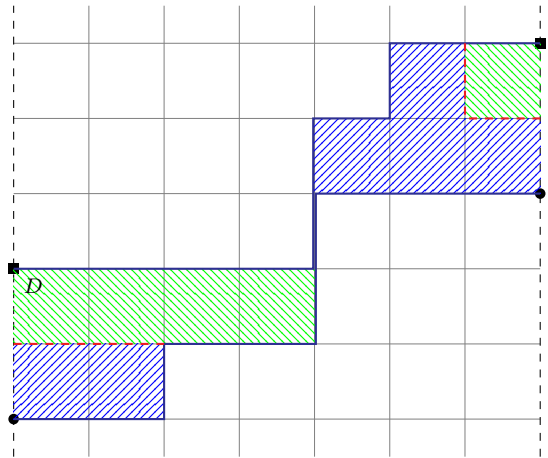
\begin{figure}[h]
\centering
\begin{tikzpicture}
\foreach \y in {0,1,...,5} {
\draw[help lines] (0 cm,\y cm) -- (7 cm,\y cm); }
\foreach \x in {1,...,6} {
\draw[help lines] (\x cm,-.5 cm) -- (\x cm,5.5 cm); }
\draw[dashed] (0,-.5 cm) -- (0,5.5cm);
\draw[dashed] (7,-.5 cm) -- (7,5.5cm);
\fill (0,0) circle (2pt);
\fill (7,3) circle (2pt);
\fill (-2pt,2cm-2pt) rectangle (2pt,2cm+2pt);
\fill (7cm-2pt,5cm-2pt) rectangle (7cm+2pt,5cm+2pt);
\fill[pattern=north west lines, pattern color=Green] (0,1) rectangle (4,2);
\fill[pattern=north west lines, pattern color=Green] (6,4) rectangle (7,5);
\fill[pattern=north east lines, pattern color=Blue] (0,0) rectangle (2,1);
\fill[pattern=north east lines, pattern color=Blue] (4,3) rectangle (7,4);
\fill[pattern=north east lines, pattern color=Blue] (5,4) rectangle (6,5);
\draw[thick,Blue] (0,0) -- (2,0) -- (2,1) -- (4cm+.5pt,1) -- (4cm+.5pt,3) -- (7,3);
\draw[thick,Blue] (0,2) -- (4cm-.5pt,2) -- (4cm-.5pt,4) -- (5,4) -- (5,5) -- (7,5);
\draw[thick,Red,dashed] (0,1) -- (2,1);
\draw[thick,Red,dashed] (6,5) -- (6,4) -- (7,4);
\node[font=\scriptsize, below right] at (0,2) {$D$}; 
\end{tikzpicture}
\caption{A $(7,3)$-periodic configuration consisting of two paths. The signature of $M(D,0)$ is $\{5,6\}$.}
\label{fig:73-2}
\end{figure}
Figure \ref{fig:73-2} shows a $(7,3)$-periodic configuration $\SL$ such that the algebra $\CA(\SL)$ has a unique finite-dimensional simple module $M(D,0)$. This module has dimension $11$ and signature $\{5,6\}$. 
\end{Example}

\section*{AI declaration}
The abstract, the introduction (Section 1), and Theorem \ref{thm:signature} with its proof and Remark \ref{rem:local-global} in this version were drafted with the assistance of Claude Opus 5.5 (Anthropic), working from the author's earlier version of the paper (arXiv:1702.08168). The tool also found a counting error in the earlier description of the signature, which is corrected in Theorem \ref{thm:signature}(iii). The mathematical results are those of the earlier version. The author has checked all AI-assisted text and arguments and takes full responsibility for the content.

\bibliographystyle{siam}

\end{document}